\documentclass[12pt]{amsart}
\usepackage{amssymb,amscd}
\usepackage{psfrag}
\usepackage{graphicx} 
\newtheorem{thm}{Theorem}[section]
\newtheorem{lem}[thm]{Lemma}
\newtheorem{prop}[thm]{Proposition}
\newtheorem{cor}[thm]{Corollary}
\newtheorem*{thm*}{Theorem}

\newcommand{\R}{{\mathbb R}}
\newcommand{\Z}{{\mathbb Z}}
\newcommand{\Q}{{\mathbb Q}}
\newcommand{\C}{{\mathbb C}}
\newcommand{\Gc}{G^\C}
\newcommand{\tr}{\mathsf{tr}}
\newcommand{\Hom}{\mathsf{Hom}}
\newcommand{\Ham}{\mathsf{Ham}}
\newcommand{\Mod}{\mathsf{Mod}(\Sigma)}
\newcommand{\Aut}{\mathsf{Aut}}
\newcommand{\Out}{\mathsf{Out}}
\newcommand{\Inn}{\mathsf{Inn}}

\newcommand{\Ad}{\mathsf{Ad}}
\newcommand{\Id}{\mathbb{I}}

\newcommand{\su}{{\mathfrak{su}(2)}}
\renewcommand{\gg}{\mathfrak{g}}
\newcommand{\SU}{{\mathsf{SU}(2)}}
\newcommand{\GL}{{\mathsf{GL}}}
\newcommand{\slt}{\mathsf{SL}(2,\C)}
\newcommand{\Uo}{\mathsf{U}(1)}

\newcommand{\Mgn}{{\Sigma_{g,n}}}
\newcommand{\MB}{{\Hom_b(\pi,\SU)/\SU}}

\newcommand{\m}{{\mathfrak M}}
\newcommand{\Ff}{{\mathcal F}}
\newcommand{\f}{\mathsf{f}}

\newcommand{\F}{\mathsf{F}}

\newcommand{\Gg}{\mathfrak{G}}
\newcommand{\gG}{\mathfrak{g}}

\newcommand{\sS}{\mathcal{S}}
\newcommand{\X}{\mathfrak{X}}
\newcommand{\Xb}{\X_b}
\newcommand{\XbC}{\X_b^\C}

\newcommand{\catquo}{/{\hskip-3pt}/}

\newcommand{\Pp}{\mathfrak{P}} 
\newcommand{\Xc}{\X^\C}

\newcommand{\In}{\mathfrak{I}_N}

\numberwithin{equation}{section}
\begin{document}

\thanks{
Goldman gratefully acknowledges partial support from
National Science Foundation grant DMS070781 and the 
Oswald Veblen Fund at the Institute for Advanced Study. 
Xia gratefully acknowledges partial support by the National Science Council, 
Taiwan with grants 96-2115-M-006-002 and 97-2115-M-006-001-MY3.}

\title[Mapping class group actions]{
Ergodicity of mapping class group actions on $\SU$-character varieties}
\author[Goldman]{William M. Goldman}
\address{
Department of Mathematics,\\
University of Maryland,\\
College Park, MD 20742 \\
\tt{wmg@math.umd.edu} ({\it Goldman}) }

\author[Xia]{Eugene Z. Xia}
\address{
Division of Mathematics\\
National Center for Theoretical Science (South)\\
Department of Mathematics\\
National Cheng-kung University\\
Tainan 701, Taiwan \\
\tt{ezxia@ncku.edu.tw}
({\it Xia})}

\date{\today}
\dedicatory{To Bob Zimmer, on his sixtieth birthday}

\begin{abstract}
Let $\Sigma$ be a compact orientable surface with genus \(g\) and
\(n\) boundary components $\partial_1,\dots, \,\partial_n$.
Let $b = (b_1, \dots, b_n)\in [-2,2]^n$.
Then the mapping class group
$\Mod$ acts on the relative 
$\SU$-character variety
$
\Xb:=\MB,
$
comprising conjugacy classes of representations
$\rho$ with $\tr(\rho(\partial_i)) = b_i$.
This action preserves a symplectic structure on the open dense smooth submanifold of $\MB$
corresponding to irreducible representations.
This subset has full measure and is connected.
In this note we use the symplectic geometry of this space to give a new proof that this action is ergodic.

\end{abstract}
\maketitle

\section{Introduction}

Let $\Sigma = \Mgn$ be a compact oriented surface of genus $g$
with $n$ boundary components  $\partial_1(\Sigma),...,\partial_n(\Sigma)$.
Choose basepoints $p_0\in\Sigma$ and $p_i\in\partial_i(\Sigma).$
Let $\pi = \pi_1(\Sigma,p_0)$ denote the fundamental group of $\Sigma$.
Choosing arcs from $p_0$ to each $p_i$ identifies each fundamental
group $\pi_1\big(\partial_i(\Sigma),p_i\big)$ with a subgroup
$\pi_1(\partial_i) \hookrightarrow \pi.$
The orientation on $\Sigma$ induces orientations
on each $\partial_i(\Sigma)$.
For each $i$, denote the positively oriented generator
of $\pi_1\big(\partial_i\Sigma\big)$
also by $\partial_i$.

The {\em mapping class group\/} $\Mod$ consists of
isotopy classes of orienta\-tion-preserving homeomorphisms of $\Sigma$
which pointwise fix each $\partial_i$.
The Dehn-Nielsen Theorem (see for example
Farb-Margalit~\cite{FarbMargalit} or Morita~\cite{DehnNielsen}),
identifies $\Mod$ with a subgroup of $\Out(\pi) := \Aut(\pi)/\Inn(\pi)$.
%

Consider a connected compact semisimple Lie group $G$.
Its complexification $\Gc$ is
the group of complex points of a semisimple linear algebraic group
defined over $\R$.
Fix a conjugacy class $B_i\subset G$
for each boundary component $\partial_i$.
Then the {\em relative representation variety\/} is
\begin{equation*}
\Hom_B(\pi, G) := \{\rho \in \Hom(\pi, G) \mid
\rho(\partial_j) \in B_j, \text{~for~}   1 \le j \le n \}.
\end{equation*}

The action of the automorphism group $\Aut(\pi)$ on $\pi$ induces an
action on $\Hom_B(\pi,\Gc)$ by composition.
Furthermore this action
descends to an action of $\Mod\subset\Out(\pi)$ on the categorical quotient or the {\em relative character variety}
\begin{equation*}
\Xc_B(G) := \Hom_B(\pi,\Gc)\catquo\Gc.
\end{equation*}
The moduli space $\Xc_B(G)$ has an invariant
dense open subset  which is a smooth complex submanifold. This subset has
an invariant complex symplectic structure $\omega^\C$,
which is algebraic with respect to the structure of
$\Xc_B(G)$ as an affine algebraic set.
The pullback  $\omega$ of the real part of this complex symplectic
structure under
\begin{equation*}
\X_B(G) := \Hom_B(\pi,G)/G \longrightarrow \Xc_B(G)
\end{equation*}
defines a symplectic structure on 
a dense open subset, 
which 
is a smooth submanifold. The smooth measure
defined by the symplectic structure  is finite
\cite{Nature, Hu1, GHJW} and
$\Mod$-invariant.
The main result of Goldman~\cite{Erg} (when $G$ has $\SU$ and $\Uo$-factors) 
and Pickrell-Xia~\cite{PickrellXia} (when $g > 1$) is:

\begin{thm*} The action of $\Mod$ on each component of $\X_B(G)$
is ergodic with respect to the measure induced by $\omega$.
\end{thm*}

The goal of this note is to give a short proof in the case that $G=\SU$.

Recently F.\ Palesi~\cite{Palesi} 
proved 
ergodicity of $\Mod$ on
$\X_B(\SU)$ when $\Sigma$ is a compact connected {\em nonorientable\/}
surface with $\chi(\Sigma) \le -2$. When $\Sigma$ is nonorientable,
the character variety fails to possess a symplectic structure (in
fact its dimension may be odd) and it would be interesting to adapt
the proof given here to the nonorientable case.

The proof given here arose from our investigation~\cite{GoldmanXia} of
ergodic properties of subgroups of $\Mod$ on character varieties. The
closed curves on $\Sigma$ play a central role. Namely, every closed
curve defines a conjugacy class of elements in $\pi$, and hence a
regular function
\begin{align*}
\Hom(\pi,\Gc) &\xrightarrow{\f_\alpha} \C \\
\rho & \longmapsto  \tr\big(\rho(\alpha)\big).
\end{align*}
for some representation $\Gc\longrightarrow\GL(N,\C)$.  These trace functions $\f_\alpha$ are $G^\C$-conjugate invariant and
results of Procesi~\cite{Procesi} imply that such functions generate the coordinate
ring $\C[\X_B(\slt)]$ of $\X_B(\slt)$.

{\em Simple\/} closed curves $\alpha$ determine elements of $\Mod$, namely the
{\em Dehn twists\/} $\tau_\alpha$. Let $S$ be a set of simple closed curves
on $\Sigma$.
Our methods apply to the subgroup $\Gamma_S\subset\Mod$
generated by $\tau_\alpha$, where $\alpha\in S$. Our proof may be summarized:
{\em if the trace functions $\f_\alpha$ generate $\C[\X_B(\slt)]$,
then the action of  $\Gamma_S$ on each component of $\X_B(\SU)$ is ergodic.}

The original proof~\cite{Erg} decomposes 
$\Sigma$ 
along a set $\Pp$ of $3g-3 + 2n$ disjoint curves
into 
\begin{equation*}
2g-2 + n = -\chi(\Sigma)  
\end{equation*}
$3$-holed spheres (a {\em pants decomposition\/}.)
The subgroup $\Gamma_\Pp$ of $\Mod$ stabilizing $\Pp$ 
is generated by Dehn twists along curves in $\Pp$.
The corresponding trace functions define a map
\begin{equation*}
\Xb \xrightarrow{\f_\Pp} [-2,2]^{\Pp} 
\end{equation*}
which is an ergodic decomposition for the 
action of $\Gamma_\Pp$. Thus any measurable 
function invariant under $\Gamma_\Pp$
must factor through $\f_\Pp$.
Changing $\Pp$ by elementary moves
on $4$-holed spheres, and a detailed
analysis in the case of $\Sigma_{0,4}$ and $\Sigma_{1,1}$
implies ergodicity under all of $\Mod$.
The present proof uses the commutative algebra
of the character ring (in particular the work
of Horowitz~\cite{Horowitz}, Magnus~\cite{Magnus}
and Procesi~\cite{Procesi}, and the identification
of the twist flows with the Hamiltonians of
trace functions~\cite{InvFuns}).
Although it is not used in \cite{Erg},
the map $\f_\Pp$ is the {\em moment map\/}
for the $\R^\Pp$-action by twist flows,
as well as the ergodic decomposition for $\Gamma_\Pp$.
Finding sets $S$ of simple curves whose trace functions generate
the character ring promises to be useful to prove
ergodicity of the  subgroup of $\Mod$ generated by
Dehn twists along elements of $S$ (Goldman-Xia~\cite{GoldmanXia}.)

In a similar direction, Sean Lawton has pointed out 
that this method of proof,
(combined with Lawton [Lawton1,Lawton2]) 
implies in at least some cases ergodicity
of $\Mod$ on the relative
$\mathsf{SU}(3)$-character varieties 
(except when $\Sigma \approx \Sigma_{0,3}$, where it is not true).

We are grateful to Sean Lawton, David Fisher and the anonymous
referee for helpful suggestions on this manuscript.

With great pleasure we dedicate this paper to Bob Zimmer. Goldman
first presented this result in Zimmer's graduate course at Harvard
University in Fall 1985. Goldman would like to express his warm
gratitude to Zimmer for the friendship, support and mathematical
inspiration he has given over many years.

\section{Simple generators for the character ring}

In the paper we restrict to the case $G = \SU$ and $\Gc = \slt$.
Conjugacy classes in $G = \SU$ are the level sets of
the {\em trace function\/}
$\SU \xrightarrow{\tr} [-2,2]$.
Thus a collection $B = (B_1,\dots,B_n)$ of conjugacy classes
in $\SU$ is precisely given by an $n$-tuple
\begin{equation*}
b = (b_1,\dots,b_n)\in [-2,2]^n.
\end{equation*}
We denote the {\em relative representation variety\/} by:
\begin{equation*}
\Hom_b(\pi,\SU) := \{\rho \in \Hom(\pi,\SU) \mid \tr\big(\rho(\partial_i)\big) = b_i \}
\end{equation*}
and its quotient, the {\em relative character variety\/} by:
\begin{equation*}
\Xb := \Hom_b(\pi,\SU)/\SU.
\end{equation*}

\begin{thm}\label{thm:simple}
There exists a finite subset $\sS\subset\pi$ corresponding
to simple closed curves on $\Sigma$ such that $\{\f_\gamma :
\gamma\in\sS\}$ generates the coordinate ring $\C[\Xb]$.
\end{thm}
\noindent 
We prove this theorem in \S\ref{sec:MHP generators} and \S\ref{sec:simpleloops}.

\subsection{Magnus-Horowitz-Procesi generators}\label{sec:MHP generators}

The following well known proposition is a direct consequence of the
work of Horowitz~\cite{Horowitz} and Procesi~\cite{Procesi}.  Compare
also Magnus~\cite{Magnus}, Newstead~\cite{Newstead} and
Goldman~\cite{TraceCoords}.

\begin{prop}\label{prop:free}
Let $\F_N$ be the free group freely generated by
$A_1,\dots, A_N$,
and let
\begin{equation*}
\X(N) := \Hom(\F_N,\slt)\catquo\slt
\end{equation*}
be its $\slt$-character variety.
Denote by $\In$ the collection of all
\begin{equation*}
I = (i_1,\dots,i_k)\in \Z^k
\end{equation*}
where
\begin{equation*}
1 \le i_1 < \dots < i_k \le N
\end{equation*}
and $k\le 3$. For $I\in\In$, define
\begin{equation*}
A_I := A_{i_1}\dots A_{i_k}
\end{equation*}
and let
\begin{align*}
\X(N) &\xrightarrow{\f_I} \C \\
[\rho] &\longmapsto \tr\big(\rho(A_I)\big)
\end{align*}
the corresponding trace functions.
Then the collection
\begin{equation*}
\{\f_I \mid I\in\In \}
\end{equation*}
generates the coordinate ring $\C[\X(N)]$.
\end{prop}
We shall refer to the coordinate ring $\C[\X(N)]$
as the {\em character ring.\/} Recall that by
definition it is the 
subring of the ring of regular functions 
\begin{equation*}
\slt^N \longrightarrow \C
\end{equation*}
consisting of $\Inn(\slt)$-invariant functions.

\subsection{Constructing simple loops}\label{sec:simpleloops}

Suppose that $\Sigma$ has genus $g \ge 0$ and $n > 0$ boundary components.
(We postpone the case when $\Sigma$ is closed, that is $n=0$, to the end of this
section.)
We suppose that $\chi(\Sigma) = 2-2g - n < 0$.
Then $\pi_1(\Sigma)$ is free of rank $N = 2g + n -1$.
We describe a presentation of 
$\pi_1(\Sigma)$ 
such that the above elements $A_I$, for $I\in\In$
can be represented by simple closed curves on $\Sigma$.
We also identify $I$ with the {\em subset\/}
\begin{equation*}
\{ i_1, \dots, i_k \} \subset  \{1, \dots, N \}.
\end{equation*}
The fundamental group 
$ 
\pi_1(\Sigma)$ admits a presentation
with generators
\begin{equation*}
A_1, \dots, A_{2g}, A_{2g+1}, \dots, A_{2g+n}
\end{equation*}
subject to the relation
\begin{equation*}
A_1 A_2 A_1^{-1} A_2^{-1} \dots
A_{2g-1} A_{2g} A_{2g-1}^{-1} A_{2g}^{-1} \dots
A_{2g+1} \dots  A_{2g+n}  \;=\; 1.
\end{equation*}
Then 
\begin{equation*}
\pi \;=\; \pi_1(\Sigma) \;\cong\; \F_{2g+n-1},  
\end{equation*}
freely generated by the set $\{A_1,\dots, A_{2g+n-1}\}$.

To represent the elements $A_I\in
\pi_1(\Sigma)$ 
explicitly as {\em simple
loops,\/} we realize $\Sigma$ as the union of a planar surface $P$ and
$g$ {\em handles\/} $H_1,\dots, H_g$. In the notation of
\cite{TraceCoords}, $P\approx \Sigma_{0,g+n}$ has $g+n$ boundary
components 
\begin{equation*}
\alpha_1,\dots,\alpha_g,\alpha_{g+1},\dots, \alpha_{g+n} 
\end{equation*}
and each handle $H_j\approx\Sigma_{1,1}$ is a one-holed torus. 
The original surface $\Sigma$ is obtained by attaching $H_j$ to $P$ along
$\alpha_j$ for $j=1,\dots,g$.

We construct the curves $A_i$, for $i=1,\dots,2g+n$ as follows.
Choose a pair of basepoints $p_j^+, p_j^-$ on each $\alpha_j$
for $j=1,\dots,g+n$. Let $\alpha_j^-$ be the oriented subarc of $\alpha_j$ from
$p_j^-$ to $p_j^+$, and $\alpha_j^+$  the corresponding subarc from $p_j^+$ to $p_j^-$.
Thus $\alpha_j \simeq \alpha_j^-\ast\alpha_j^+$ is a boundary component of $P$.

Choose a system of disjoint arcs $\beta_j$ from $p_j^+$ to $p_{j+1}^-$,
where $\beta_{g+n}$ runs from $p_{g+n}^+$ to $p_{1}^-$ in the {\em cyclic indexing}
of $\{1,2,\dots,g+n\}$. Compare Figure~\ref{fig:planar3}.

For $I\in\In$, the curve $A_I$ will be the concatenation 
$E_1^I \ast \dots E_{g+n}^I$ of simple arcs $E_j^I$ 
running from $p_j^-$ to $p_{j+1}^-$.
Define
\begin{equation*}
E_j^\emptyset \;:=\; \alpha_j^-\ast \beta_j,
\end{equation*}
so that 
\begin{equation*}
A^{\emptyset}  \;:=\; E_1^\emptyset\ast \dots\ast E_N^\emptyset 
\end{equation*}
is a contractible loop.

Suppose first that $i> 2g$. 
Then the curve $A_i$ will be freely homotopic to the oriented loop
$\alpha_i^{-1}$, corresponding to a component of $\partial\Sigma$. 
The arc
\begin{equation*}
E_i^+ \;:=\; (\alpha_i^+)^{-1}\ast \beta_i
\end{equation*}
goes from $p_i^-$ to $p_{i+1}^-$ (cyclically).
Then $A_i$ corresponds to the arc
\begin{align*}
A_i  \; := \;  & E_1^\emptyset \ast \dots \ast E_{2g}^\emptyset \\
 & \ast E_{2g+1}^\emptyset \ast \dots \ast E_i^+ \ast \dots E_{2g+n-1}^\emptyset 
\end{align*}

%
%

For $i\le 2g$, the curves $A_i$ will lie on the handles $H_j$.
The curves $A_{2j-1}$ and $A_{2j}$ define a basis for the relative homology 
of $H_j$ and 
the relative homology class of the curve 
\begin{equation*}
A_{2j-1,2j} := A_{2j-1}A_{2j} 
\end{equation*}
is their sum. 
Compare Figures~\ref{fig:handle2},\ref{fig:handle3}.

As above we define three simple arcs $\gamma_j,\delta_j,\eta_j$ running from
$p_j^-$ to $p_j^+$ to build these three curves respectively.

The boundary $\partial H_j$ identifies with $\alpha_j$ for $j=1,\dots, g$.
The two points on $\partial H_j$ which identify to
\begin{equation*}
p_j^{\pm}\in\alpha_j\subset\partial P
\end{equation*}
divide $\partial H_j$ into two arcs.
Without danger of confusion, denote these arcs by $\alpha_j^{\pm}$ as well.
On the handle $H_j$, choose disjoint simple arcs $\gamma_j$, $\delta_j$ and $\eta_j$ running from
$p_i^+$ to $p_i^-$ such that the
\begin{equation*}
H_j \setminus (\gamma_j\cup \delta_j)
\end{equation*}
is an hexagon.
Two of its edges correspond to the arcs $\alpha_j^\pm$. Its other four edges are
the two pairs obtained by splitting $\gamma_j$ and $\delta_j$.
(Compare Figure~\ref{fig:handle2}.)
Let $\eta_j$ to be a simple arc homotopic to $\gamma_j\ast (\alpha_j^+)^{-1}\ast \delta_j$, 
where $\ast$ denotes concatenation.
For each $j\le g$, the arcs
\begin{align*}
E_j^\gamma &\;=\; \gamma_j \ast \beta_j \\
E_j^\delta &\;=\; \delta_j \ast \beta_j \\
E_j^\eta &\;=\; \eta_j \ast \beta_j 
\end{align*}
run from $p_j^-$ to $p_j^+$ and define:
\begin{align*}
A_{2j-1} &\;=\; E_1^\emptyset \ast \dots \ast E_j^\gamma \ast \dots E_g^\emptyset \ast \dots \ast E_{g+n}^\emptyset \\
A_{2j} &\;=\; E_1^\emptyset \ast \dots \ast E_j^\delta \ast \dots E_g^\emptyset \ast \dots \ast E_{g+n}^\emptyset  \\
A_{2j-1,2j} &\;=\; E_1^\emptyset \ast \dots \ast E_j^\eta \ast \dots E_g^\emptyset \ast \dots \ast E_{g+n}^\emptyset.
\end{align*}
In general, suppose that $I \in \In$.
Define
\begin{equation*}
A_I := E_1^I \ast \dots \ast E_{g+n} 
\end{equation*}
where
\begin{equation*}
E_j^I = \begin{cases}
E_j^\emptyset & \text{ if } j\notin I \\
E_j^+ & \text{ if } j\in I 
\end{cases}\end{equation*}
if $j>g$ and
\begin{equation*}
E_j^I = \begin{cases}
E_j^\emptyset & \text{ if } 2j-1, 2j\notin I \\
E_j^\gamma & \text{ if } 2j-1\in I,\, 2j\notin I \\
E_j^\delta & \text{ if } 2j-1\notin I,\, 2j\in I \\
E_j^\eta & \text{ if } 2j-1,2j \in I.
\end{cases}\end{equation*}
if $j\le g$.

Now each $A_I$ is {\em simple:\/} 
Each of the oriented arcs $\alpha_i^\pm$, $\beta_i, \gamma_i,\delta_i,\eta_i$ are embedded
and intersect only along $p_i^{\pm}$. In particular each of the above oriented arcs 
begins at some $p_i^\pm$ and ends at some $p_i^\mp$.
Thus each 
\begin{equation*}
E_j^\emptyset,  E_j^+, 
E_j^\gamma, E_j^\delta, E_j^\eta
\end{equation*}
is a simple arc running from $p_j^-$ to $p_{j+1}^-$, cyclically.
The loop $A_I$ concatenates these arcs, which only intersect
along the $p_i^-$. Each of these endpoints occurs exactly twice, once as the initial
endpoint and once as the terminal endpoint.
Therefore the loop $A_I$ is simple.

This collection $A_I$, for $I\in\In$, 
of simple curves determines a collection of regular functions
$\f_I$ on $\Xc$ which generate the character ring. 
Since the inclusion 
\begin{equation*}
\XbC \hookrightarrow \Xc
\end{equation*}
is a morphism of algebraic sets, the restrictions
of $\f_I$ to $\XbC$ generate the coordinate ring of $\XbC$.

The case $n=0$ remains. To this end, the character variety of $\Sigma_{g,0}$ appears as the
relative character variety of $\Sigma_{g,1}$ with boundary condition $b_1=2$.
As above, the restrictions of the $\f_I$ from $\Sigma_{g,1}$ to the character variety
of $\Sigma_{g,0}$ generate its coordinate ring. The proof
of Theorem~\ref{thm:simple} is complete.


 \clearpage
\begin{figure}[!h]
\centering
\psfrag{a1}{$A_1$}
\psfrag{a2}{$A_2$}
\psfrag{a3}{$A_3$}
\psfrag{a4}{$A_4$}
\psfrag{a1a2}{$A_1A_2$}
\psfrag{a1a3}{$A_1A_3$}
\psfrag{a1a2a3}{$A_1A_2A_3$}
\psfrag{a2a3}{$A_2A_3$}
\includegraphics[scale=.5]
{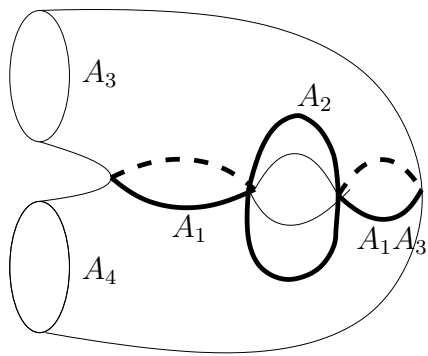}
\caption{Simple loops on $\Sigma_{1,2}$
corresponding to  words $A_1,A_2,A_3,A_1A_3$ 
and $A_4^{-1} = A_1A_2A_1^{-1}A_2^{-1}A_3$
in free generators $\{A_1,A_2,A_3\}$.}
\label{fig:twoholedtorus1}
\end{figure}

\begin{figure}[h]
\centering
\psfrag{a1}{$A_1$}
\psfrag{a2}{$A_2$}
\psfrag{a3}{$A_3$}
\psfrag{a4}{$A_4$}
\psfrag{a1a2}{$A_1A_2$}
\psfrag{a1a3}{$A_1A_3$}
\psfrag{a1a2a3}{$A_1A_2A_3$}
\psfrag{a2a3}{$A_2A_3$}
\includegraphics[scale=.25]{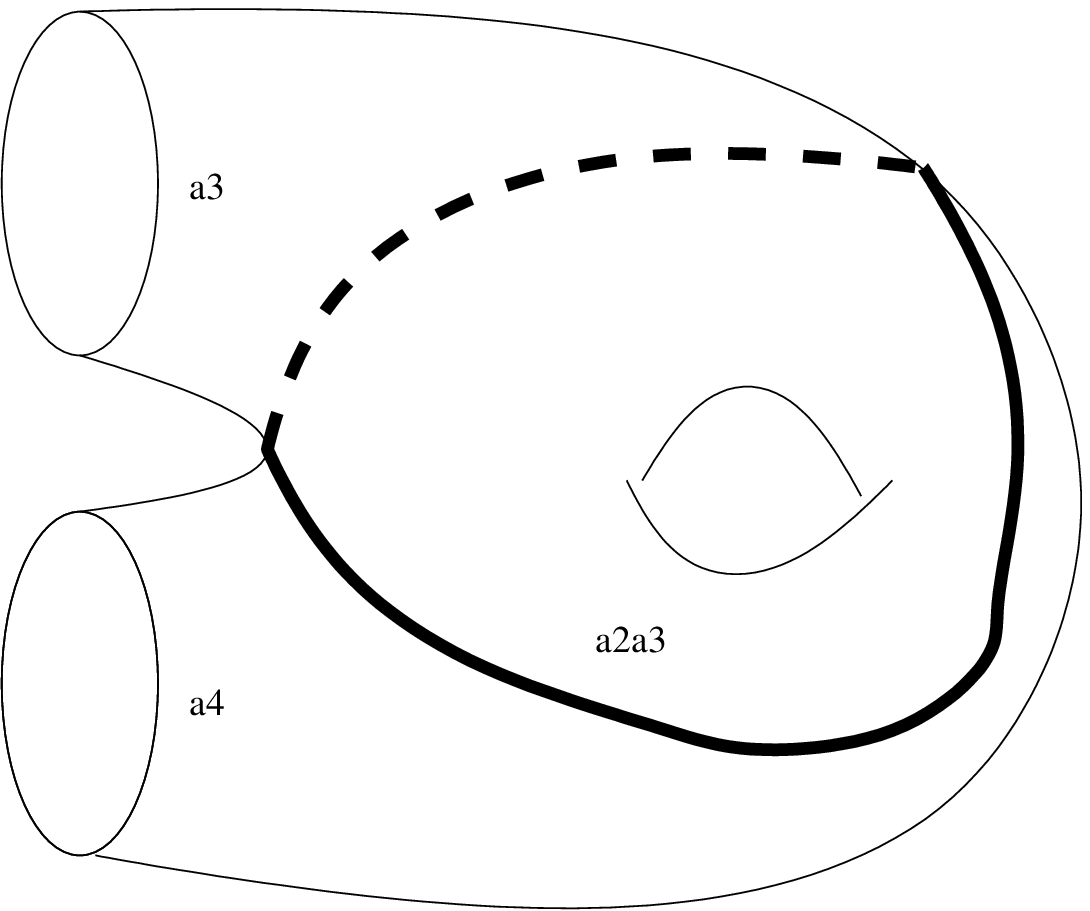}
\caption{Simple loop 
corresponding to $A_2A_3$.}
\label{fig:twoholedtorus2}
\end{figure}

\begin{figure}[h]
\centering
\psfrag{a1}{$A_1$}
\psfrag{a2}{$A_2$}
\psfrag{a3}{$A_3$}
\psfrag{a4}{$A_4$}
\psfrag{a1a2}{$A_1A_2$}
\psfrag{a1a3}{$A_1A_3$}
\psfrag{a1a2a3}{$A_1A_2A_3$}
\psfrag{a2a3}{$A_2A_3$}
\includegraphics[scale=.25]{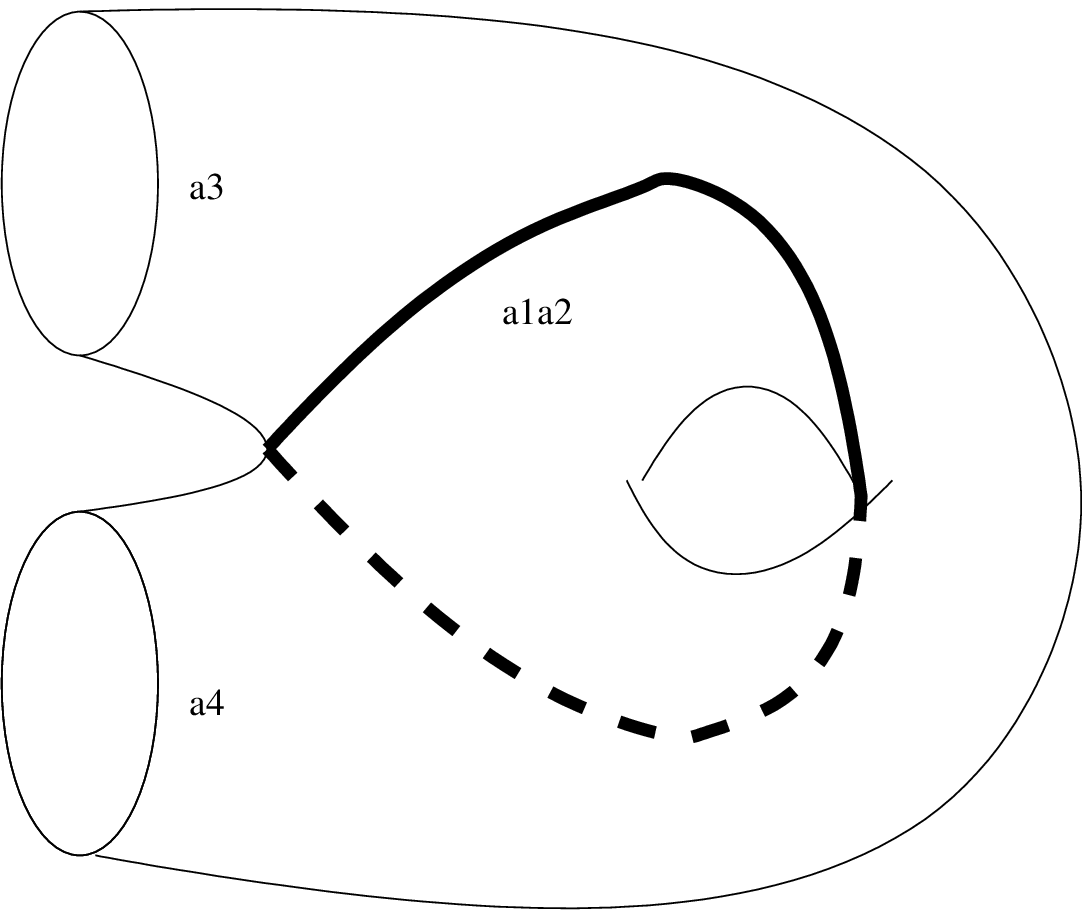}
\caption{Simple loop 
corresponding to  $A_1A_2$.}
\label{fig:twoholedtorus3}
\end{figure}

\begin{figure}[h]
\centering
\psfrag{a1}{$A_1$}
\psfrag{a2}{$A_2$}
\psfrag{a3}{$A_3$}
\psfrag{a4}{$A_4$}
\psfrag{a1a2}{$A_1A_2$}
\psfrag{a1a3}{$A_1A_3$}
\psfrag{a1a2a3}{$A_1A_2A_3$}
\psfrag{a2a3}{$A_2A_3$}
\includegraphics[scale=.25]{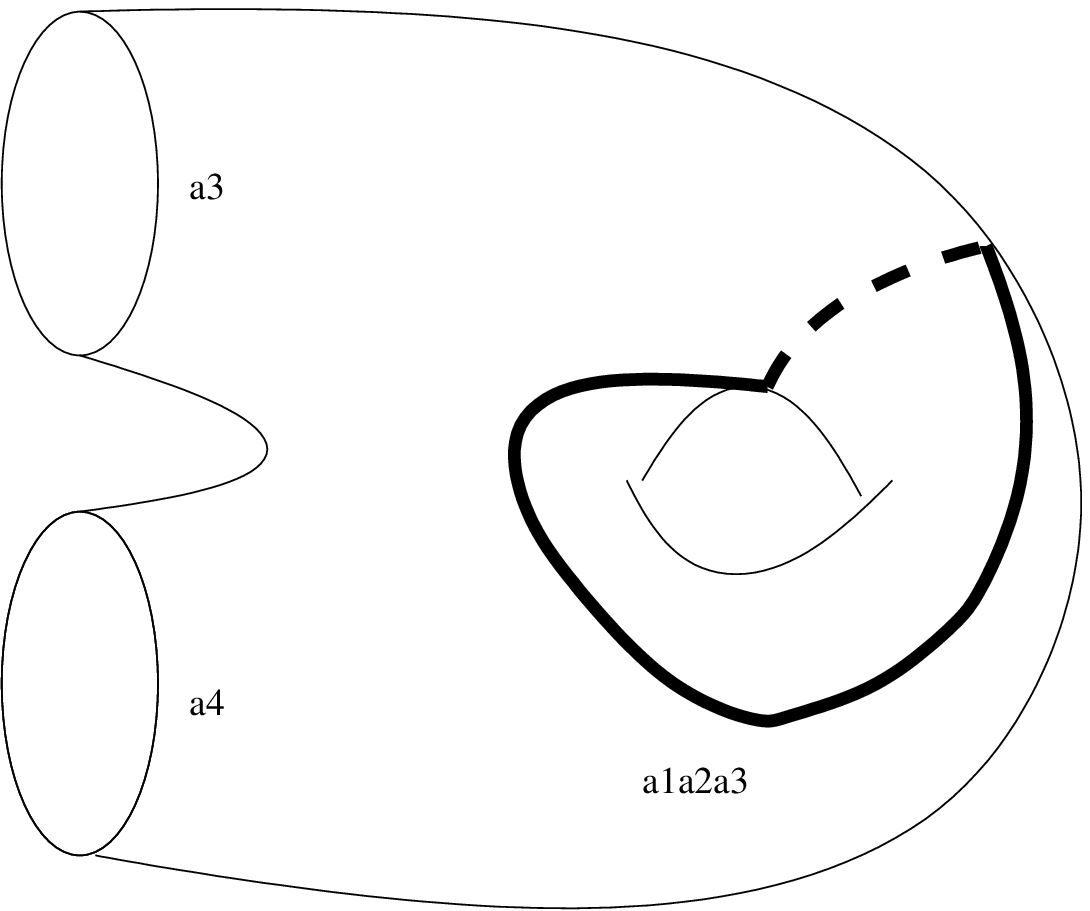}
\caption{Simple loop corresponding to  $A_1A_2A_3$.}
\label{fig:twoholedtorus4}
\end{figure}
\clearpage

\begin{figure}[!hbp]
\centering
\psfrag{p1+}{$p_1^-$} 
\psfrag{p1-}{$p_1^+$}
\psfrag{p2-}{$p_2^+$}
\psfrag{p2+}{$p_2^-$}
\psfrag{p3-}{$p_3^+$}
\psfrag{p3+}{$p_3^-$}
\psfrag{p4-}{$p_4^-$} 
\psfrag{p4+}{$p_4^+$} 

\psfrag{d1}{$\alpha_1$}
\psfrag{d2}{$\alpha_2$}
\psfrag{d3}{$\alpha_3$}
\psfrag{d4}{$\alpha_4$}
\psfrag{a1+}{$\alpha_1^+$}
\psfrag{a2+}{$\alpha_2^+$}
\psfrag{a3+}{$\alpha_3^+$}
\psfrag{a4+}{$\alpha_4^+$}
\psfrag{a1-}{$\alpha_1^-$}
\psfrag{a2-}{$\alpha_2^-$}
\psfrag{a3-}{$\alpha_3^-$}
\psfrag{a4-}{$\alpha_4^-$}
\psfrag{b1}{$\beta_1$}
\psfrag{b2}{$\beta_2$}
\psfrag{b3}{$\beta_3$}
\psfrag{b4}{$\beta_4$}
\includegraphics[scale=.7]{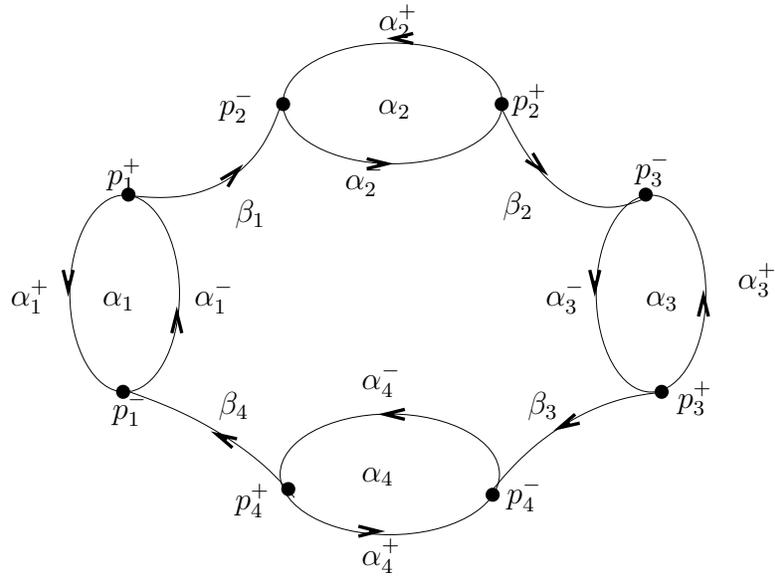}
\caption{A planar surface $P \approx \Sigma_{0,4}$}
\label{fig:planar3}
\end{figure}

\begin{figure}[!h]
\centering
\psfrag{a+}{$\alpha_j^+$}
\psfrag{a-}{$\alpha_j^-$}
\psfrag{p+}{$p_j^+$}
\psfrag{p-}{$p_j^-$}
\psfrag{c}{$\gamma_j$}
\psfrag{d}{$\delta_j$}
\includegraphics[scale=.5]{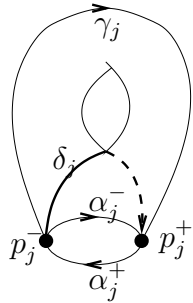}
\caption{A handle $H_j \approx\Sigma_{1,1}$}
\label{fig:handle2}
\end{figure}

\begin{figure}[!h]
\centering
\psfrag{a+}{$\alpha_j^+$}
\psfrag{a-}{$\alpha_j^-$}
\psfrag{p+}{$p_j^+$}
\psfrag{p-}{$p_j^-$}
\psfrag{e}{$\eta_j$}
\includegraphics[scale=.5]{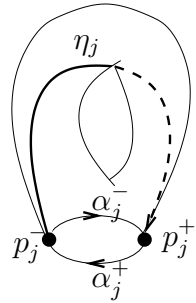}
\caption{A $(1,1)$-curve $\eta_j$ on the handle $H_j$}
\label{fig:handle3}
\end{figure}
\clearpage

\section{Infinitesimal transitivity}

The application of Theorem~\ref{thm:simple} involves several
lemmas to deduce that the flows of the Hamiltonian vector fields
$\Ham(\f_\gamma)$, where $\gamma\in\sS$, generate a transitive
action on $\Xb$.

\begin{lem}\label{lem:span}
Let $X$ be an affine variety over a field $k$.
Suppose that $\Ff\subset k[X]$ generates the coordinate ring $k[X]$.
Let $x\in X$.
Then the differentials $df(x)$, for $f \in \Ff$, span the cotangent space $T^*_x(X)$.
\end{lem}
\begin{proof}
Let $\m_x\subset k[X]$ be the maximal ideal corresponding to $x$.
Then the functions $f - f(x) 1$, where $f\in\Ff$, span $\m_x$.
The correspondence
\begin{align*}
\m_x &\longrightarrow T_x^*(X) \\
f &\longmapsto df(x)
\end{align*}
induces an isomorphism
$\m_x/\m_x^2\xrightarrow{\cong} T_x^*(X)$.
In particular it is onto.
Therefore the covectors $df(x)$ span $T_x^*(X)$ as claimed.
\end{proof}

\begin{lem}\label{lem:Ham}
Let $X$ be a connected symplectic manifold and $\Ff$ be a set of
functions on $X$ such that at every point $x\in X$, the
differentials $df(x)$, for $f\in\Ff$, span the cotangent space
$T_x^*(X)$. Then the group $\Gg$ generated by the Hamiltonian flows
of the vector fields $\Ham(f)$, for $f\in\Ff$, acts
transitively on $X$.
\end{lem}
\begin{proof}
The nondegeneracy of the symplectic structure implies that the
vector fields $\Ham(f)(x)$ span the tangent space $T_xX$ for every $x\in X$.
By the inverse function theorem, the $\Gg$-orbit $\Gg\cdot x$ of $x$ is open.
Since the orbits partition $X$ and $X$ is connected,
$\Gg\cdot x = X$ as claimed.
\end{proof}

\begin{prop}\label{prop:connected}
Let $b = (b_1,\dots,b_m)\in [-2,2]^n$.
Then $\Xb$ is either empty or connected.
\end{prop}
The proof follows from Newstead~\cite{Newstead1} and Goldman~\cite{TopComps}.
Alternatively, apply Mehta-Seshadri~\cite{MehtaSeshadri} to identify $\Xb$
with a moduli space of semistable parabolic bundles, and apply their
result that the corresponding moduli space is irreducible.

\begin{cor}\label{cor:transitiveHam}
Let $\Gg$ be the group generated by the flows of
the Hamiltonian vector fields $\Ham(\f_\gamma)$, where $\gamma\in\sS$.
Then $\Gg$ acts transitively on $\Xb$.
\end{cor}
\begin{proof}
By Theorem~\ref{thm:simple},
\begin{equation*}
\{ f_\gamma \mid \gamma\in\sS \}
\end{equation*}
generates $\C[\Xb]$.
Lemma~\ref{lem:span} implies that at every point $x\in\Xb$ the differentials
$d\f_\gamma(x)$ span $T^*_x(\Xb)$.
Proposition~\ref{prop:connected} implies that $\Xb$ is connected.
Now apply Lemma~\ref{lem:Ham}.
\end{proof}

\section{Hamiltonian twist flows}
We briefly review results of Goldman~\cite{InvFuns},
describing the flows generated by
the Hamiltonian vector fields $\Ham(\f_\alpha)$,
when $\alpha$ represents a {\em simple\/}  closed curve.
In that case the local flow of this vector field on $\X(G)$
lifts to a flow $\xi_t$ on the representation variety $\Hom_B(\pi,G)$.
Furthermore this flow admits a simple description~\cite{InvFuns}
as follows.

\subsection{
Invariant functions and centralizing one-parameter subgroups}
Let  $\Ad$ be the adjoint representation
of $G$ on its Lie algebra $\gg$.
We suppose that $\Ad$ preserves a nondegenerate symmetric
bilinear form $\langle,\rangle$ on $\gg$. In the case $G=\SU$, this
will be
\begin{equation*}
\langle X, Y\rangle := \tr (X Y).
\end{equation*}

Let $G\xrightarrow{\f} \R$ be a function invariant under the inner automorphisms $\Inn(G)$.
Following \cite{InvFuns}, we describe how
$\f$ determines a way to associate to every element $x\in G$
a one-parameter subgroup
\begin{equation*}
\zeta^t(x) \, = \, \exp\big( t \F(x) \big)
\end{equation*}
centralizing $x$.
Given  $\f$, define  its {\em variation function\/}
$ G \xrightarrow{\F} \gG$ by:
\begin{equation*}
\langle \F(x), \upsilon \rangle = \frac{d}{dt}\bigg|_{t=0}
\f \big( x \exp(t\upsilon)\big)
\end{equation*}
for all $\upsilon\in\gG$.
Invariance of $\f$ under $\Ad(G)$ implies that $\F$ is $G$-equivariant:
\begin{equation*}
\F( g x g^{-1}) = \Ad(g) \F(x).
\end{equation*}
Taking $g = x$ implies that the one-parameter subgroup
\begin{equation}\label{eq:oneparameter}
\zeta^t(x) := \exp(t \F(x))
\end{equation}
lies in the centralizer of $x\in G$.

Intrinsically, $\F(x)\in\gG$ is dual (by  $\langle,\rangle$) to the
element of $\gG^*$ corresponding to
the left-invariant 1-form  on $G$ extending the covector
$df(x) \in T^*_x(G)$.

\subsection{Nonseparating loops}
There are two cases, depending on whether
$\alpha$ is {\em nonseparating\/} or {\em separating.\/}
Let $\Sigma | \alpha$ denote the surface-with-boundary obtained by
{\em splitting\/} $\Sigma$ along $\alpha$. The boundary of
$\Sigma | \alpha$ has two components, denoted by $\alpha_\pm$, corresponding
to $\alpha$. The original surface $\Sigma$ may be reconstructed as a quotient
space under the identification of $\alpha_-$ with $\alpha_+$.

If $\alpha$ is nonseparating, then $\pi = \pi_1(\Sigma)$ can be
reconstructed from the fundamental group $\pi_1(\Sigma | \alpha)$ as
an HNN-extension:
\begin{equation}\label{eq:hnn}
\pi \;\cong\;   \bigg(\pi_1(\Sigma | \alpha) \amalg
\langle\beta\rangle \bigg)
\bigg/
\bigg(\beta \alpha_- \beta^{-1} = \alpha_+ \bigg).
\end{equation}
A representation $\rho$ of $\pi$ is determined by:
\begin{itemize}
\item the restriction $\rho'$ of $\rho$ to the subgroup
$\pi_1(\Sigma|\alpha)\subset\pi$, and
\item the value $\beta' = \rho(\beta)$
\end{itemize}
which satisfies:
\begin{equation}\label{eq:hnnrep}
\beta'  \rho'(\alpha_-) \beta'^{-1} = \rho'(\alpha_+).
\end{equation}
Furthermore any pair $(\rho',\beta')$ where $\rho'$ is a representation of
$\pi_1(\Sigma|\alpha)$ and $\beta'\in G$ satisfies \eqref{eq:hnnrep}
determines a representation $\rho$ of $\pi$.

The {\em twist flow\/} $\xi_\alpha^t$, for $t\in\R$ on $\Hom(\pi,\SU)$, is
then defined as follows:
\begin{equation}\label{eq:hamtwist1}
\xi_\alpha^t(\rho):\gamma \longmapsto \begin{cases}  \rho(\gamma)
& \text{if}~ \gamma \in \pi_1(\Sigma|\alpha) \\
\rho(\beta) \zeta^t\big(\rho(\alpha_-)\big)
& \text{if}~ \gamma = \beta. \end{cases}
\end{equation}
where $\zeta^t$ is defined in \eqref{eq:oneparameter}.
This flow covers the flow generated
by $\Ham(\f_\alpha)$ on $\Xb$
(See \cite{InvFuns}).

\subsection{Separating loops}
If $\alpha$ separates, then $\pi = \pi_1(\Sigma)$ can be
reconstructed from the fundamental groups $\pi_1(\Sigma_i)$
of the two components $\Sigma_1, \Sigma_2$ of $\Sigma | \alpha$,
as an amalgam
\begin{equation}\label{eq:amalgam}
\pi \cong  \pi_1(\Sigma_1) \amalg_{\langle\alpha\rangle} \pi_1(\Sigma_2).
\end{equation}
A representation $\rho$ of $\pi$ is determined by its restrictions
$\rho_i$ to $\pi_1(\Sigma_i)$.
Furthermore any two representations $\rho_i$ of $\pi$
satisfying $\rho_1(\alpha) = \rho_2(\alpha)$ determines a representation
of $\pi$.

The {\em twist flow\/} is defined by:
\begin{equation}\label{eq:hamtwist2}
\xi_\alpha^t(\rho):\gamma \longmapsto \begin{cases}  \rho(\gamma)
& \text{if}~ \gamma \in \pi_1(\Sigma_1) \\
\zeta^t\big(\rho(\alpha)\big) \
\rho(\gamma) \
\zeta^{-t}\big(\rho(\alpha)\big)
& \text{if}~ \gamma \in \pi_1(\Sigma_2)
\end{cases}
\end{equation}
where $\zeta^t$ is defined in \eqref{eq:oneparameter}.


\subsection{Dehn twists}
Let  $\alpha\subset\Sigma$ be a simple closed curve.
The {\em Dehn twist\/} along $\alpha$ is the mapping class
$\tau_\alpha\in\Mod$ represented by a homeomorphism $\Sigma\longrightarrow\Sigma$ supported in a tubular neighborhood
$N(\alpha)$ of $\alpha$ defined as follows.  In terms of a homeomorphism
$S^1 \times [0,1] \xrightarrow{h} N(\alpha) $
which takes $\alpha$  to $S^1\times \{0\}$, the Dehn twist is
\begin{equation*}
\tau_\alpha \circ h \big(\zeta,t) = h(e^{2t \pi i} \zeta, t).
\end{equation*}
If $\alpha$ is essential, then $\tau_\alpha$ induces a nontrivial
element of $\Out(\pi)$ on $\pi=\pi_1(\Sigma)$.

If $\alpha$ is nonseparating, then $\pi = \pi_1(\Sigma)$ can be
reconstructed from the fundamental group $\pi_1(\Sigma | \alpha)$ as
an HNN-extension as in \eqref{eq:hnn}.
The Dehn twist $\tau_\alpha$ induces the automorphism
$(\tau_\alpha)_*\in\Aut(\pi)$ defined by:

\begin{equation*}
(\tau_\alpha)_*:\gamma \longmapsto
\begin{cases}  \gamma
& \text{if}~ \gamma \in \pi_1(\Sigma|\alpha) \\
\gamma  \alpha
& \text{if}~ \gamma = \beta. \end{cases}
\end{equation*}

The induced map $(\tau_\alpha)^*$ on $\Hom(\pi,G)$ maps
$\rho$ to:
\begin{equation}\label{eq:dehnnonsep}
(\tau_\alpha)^*(\rho):\gamma \longmapsto
\begin{cases}  \rho(\gamma)
& \text{if}~ \gamma \in \pi_1(\Sigma|\alpha) \\
\rho(\gamma)  \rho(\alpha)^{-1}
& \text{if}~ \gamma = \beta. \end{cases}
\end{equation}

If $\alpha$ separates, then $\pi = \pi_1(\Sigma)$ can be
reconstructed from the fundamental groups $\pi_1(\Sigma_i)$
as an amalgam as in \eqref{eq:amalgam}.
The Dehn twist $\tau_\alpha$ induces the automorphism
$(\tau_\alpha)_*\in\Aut(\pi)$ defined by:
\begin{equation*}
(\tau_\alpha)_*:\gamma \longmapsto
\begin{cases}  \gamma
& \text{if}~ \gamma \in \pi_1(\Sigma_1) \\
\alpha\gamma\alpha^{-1}
& \text{if}~ \gamma \in \pi_1(\Sigma_2)
\end{cases}.
\end{equation*}
The induced map $(\tau_\alpha)^*$ on $\Hom(\pi,G)$ maps
$\rho$ to:
\begin{equation}\label{eq:dehnsep}
(\tau_\alpha)^*(\rho):\gamma \longmapsto
\begin{cases}
\rho(\gamma)
& \text{if}~ \gamma \in \pi_1(\Sigma_1) \\
\rho(\alpha)^{-1}\rho(\gamma)  \rho(\alpha)
& \text{if}~ \gamma \in \pi_1(\Sigma_2).
\end{cases}\end{equation}

\section{The case $G=\SU$}
Now we specialize the preceding theory to
the case $G = \SU$.
Its Lie algebra $\su$ consists of $2\times 2$
traceless skew-Hermitian matrices over $\C$.

\subsection{One-parameter subgroups}
The trace function
\begin{align*}
\SU &\xrightarrow{\f} [-2,2] \\
x &\longmapsto  \tr(x)
\end{align*}
induces the variation function
\begin{align*}
\SU &\xrightarrow{\F} \su \\
x &\longmapsto  x - \frac{\tr(x)}2 \Id,
\end{align*}
the projection of $x\in\SU\subset\mathsf{M}_2(\C)$
to $\su$.
Explicitly, if $x\in\SU$, there exists $g\in\SU$ such that
\begin{equation*}
x = g \bmatrix e^{i\theta} & 0 \\ 0 & e^{-i\theta}  \endbmatrix g^{-1}.
\end{equation*}
Then $\f(x) = 2\cos(\theta)$,
\begin{equation*}
\F(x)  = g \bmatrix 2i\sin(\theta) & 0 \\ 0 & -2i\sin(\theta)  \endbmatrix g^{-1}
\in\su
\end{equation*}
and the corresponding one-parameter subgroup is
\begin{equation*}
\zeta^t(x) = g \bmatrix e^{2i\sin(\theta)t} & 0 \\ 0 & e^{-2i\sin(\theta)t} \endbmatrix g^{-1}
\in\SU
\end{equation*}
Except in two exceptional cases this one-parameter subgroup is isomorphic to $S^1$.
Namely,  $f(x)=\pm 2$, then $x = \pm \Id$. These comprise the {\em center\/} of $\SU$.
In all other cases,  $-2 < f(x)< 2$ and $\zeta^t(x)$ is a circle subgroup.
Notice that this circle subgroup contains $x$:
\begin{equation}\label{eq:embedDehnInFlow}
x \;=\; \zeta^{s(x)}(x)
\end{equation}
where
\begin{equation}\label{eq:exponentF}
s(x) \;:= \; \frac{2}{\sqrt{4-\f(x)^2}}\  \cos^{-1}\bigg( \frac{\f(x)}{2}\bigg).
\end{equation}
Furthermore
\begin{equation}\label{eq:period}
\zeta^t(x) = \Id
\end{equation}
if and only if
\begin{equation*}
t\;\in\; \frac{4\pi}{\sqrt{4-\f(x)^2}}\, \Z.
\end{equation*}
(Compare Goldman~\cite{Erg}.)

\begin{prop}\label{prop:DehnAndTwist}
Let $\alpha\in\pi$ be represented by a simple closed curve,
and $\xi_\alpha^t$ be the corresponding twist flow
on $\Hom(\pi,G)$ as defined in
\eqref{eq:hamtwist1} and \eqref{eq:hamtwist2}. 
Let $\rho\in\Hom(\pi,G)$.
Then
\begin{equation*}
(\tau_\alpha)^*(\rho) = \xi_\alpha^{s\big(\rho(\alpha)\big)}
\end{equation*}
where $s$ is defined in \eqref{eq:exponentF}.
\end{prop}
\begin{proof}
Combine \eqref{eq:embedDehnInFlow} with
\eqref{eq:hamtwist1} when $\alpha$ is nonseparating case and
\eqref{eq:hamtwist2} when $\alpha$ separates.
\end{proof}

The basic dynamical ingredient of our proof,
(like the original proof in \cite{Erg} is the
ergodicity of an irrational rotation of
$S^1$. There is a unique translation-invariant
probability measure on $S^1$ (Haar measure).
Furthermore this measure is ergodic under
the action of any infinite cyclic subgroup.
Recall that an action of  group $\Gamma$ of 
measure-preserving transformations 
of a measure space $(X,\mathcal{B},\mu)$ 
is {\em ergodic\/}
if and only if
every invariant measurable set has either
measure zero or has full measure 
(its complement has measure zero).

\begin{lem}\label{lem:irrational}
If $\cos^{-1}\big( \f(x)/2\big)/\pi$ is irrational,
then the cyclic group $\langle x\rangle$ is a
dense subgroup of the one-parameter subgroup
\begin{equation*}
\{\zeta^t(x)\mid t\in\R\}\cong S^1
\end{equation*}
and acts ergodically on $S^1$
with respect to Lebesgue measure. 
\end{lem}
For these basic facts see 
Furstenberg~\cite{Furstenberg},
Haselblatt-Katok~\cite{HasselblattKatok},
Morris~\cite{Morris} or Zimmer~\cite{Zimmer}.

\begin{cor}\label{cor:DehnTwistErgodic}
Let $\alpha, \xi_\alpha^t$ and $\tau_\alpha$
be as in Proposition~\ref{prop:DehnAndTwist}.
Then for almost every $b\in[-2,2]$,
$(\tau_\alpha)^*$ acts ergodically on
the orbit
\begin{equation*}
\{\xi_\alpha^t([\rho])\}_{t\in\R},
\end{equation*}
when $\f_\alpha(\rho) = b$.
\end{cor}
\begin{proof}
Combine Proposition~\ref{prop:DehnAndTwist}
with Lemma~\ref{lem:irrational}.
\end{proof}


\begin{prop}\label{lem:discreteflow}
Let $\alpha \in S$ be a simple closed curve, with twist vector field
$\xi_\alpha$ and Dehn twist
$\tau_\alpha$. Let $\Xb \xrightarrow{\psi}  \R$ be a
measurable function invariant under the cyclic group
$\langle (\tau_\alpha)^* \rangle$.
Then there exists a nullset $\mathcal{N}$ of $\Xb$ such that the restriction of $\psi$
to the complement of $\mathcal{N}$ is constant on each
orbit of the twist flow $\xi_\alpha$.
\end{prop}

\begin{proof}
Disintegrate the symplectic measure on $\Xb$ over the quotient map
\begin{equation*}
\Xb \longrightarrow \Xb/{\xi_\alpha}
\end{equation*}
as in Furstenberg~\cite{Furstenberg} or Morris~\cite{Morris}, 3.3.3,
3.3.4.  By \eqref{eq:period} almost all fibers of this map are
circles.

The subset
\begin{equation*}
\mathcal{N} := \f_\alpha^{-1} \big( 2 \cos(\Q \pi)\big) \subset \Xb
\end{equation*}
has measure zero.
By Corollary~\ref{cor:DehnTwistErgodic},
the action of $\big(\tau_\alpha\big)^*$ is ergodic on each circle
in the complement of $\mathcal{N}$.
In particular, $\psi$ factors through the quotient map, as desired.
\end{proof}

\begin{proof}[Conclusion of proof of Main Theorem]
Suppose that $\Xb\xrightarrow{\psi}\R$ is a measurable function
invariant under $\Mod$;
we show that $\psi$ is almost everywhere constant.

To this end let $\sS$ be the collection of simple closed curves in
Theorem~\ref{thm:simple}.
Then, for each $\alpha\in\sS$, the function
$\psi$ is invariant under the mapping
$(\tau_\alpha)^*$ induced by the Dehn twist along $\alpha\in\sS$.
By Proposition~\ref{lem:discreteflow}, $\psi$ is constant along
almost every orbit of the Hamiltonian flow $\xi_\alpha$ of $\Ham(\f_\alpha)$.
Thus, up to a nullset, $\psi$ is constant along the orbits
of the group $\Gg$ generated by these flows.
By Corollary~\ref{cor:transitiveHam}, $\Gg$ acts transitively on $\Xb$.
Therefore $\psi$ is almost everywhere constant, as claimed.
The proof is complete.
\end{proof}

\end{document}